\documentclass[a4paper,12pt,twoside,reqno]{amsart}
\usepackage{latexsym,amsmath,amsthm,verbatim,ifthen,amssymb,graphicx,color,mathrsfs}
\usepackage[all]{xy}
\usepackage{color}
\usepackage{graphicx}
\usepackage{wrapfig}

\SelectTips{cm}{}
\newtheorem{theorem}{Theorem}[section]

\newtheorem{corollary}[theorem]{Corollary}
 \newtheorem{lemma}[theorem]{Lemma}
 \newtheorem{proposition}[theorem]{Proposition}
 \theoremstyle{definition}
 \newtheorem{definition}[theorem]{Definition}
 \theoremstyle{remark}
 \newtheorem{remark}[theorem]{Remark}
 \newtheorem{example}[theorem]{Example}

\newcommand{\bq}{\mathbb Q}

\def\Q{\mathbb{Q}}
\def\R{\mathbb{R}}
\def\bc{\mathbb{C}}

\def\im{{\rm Im\,}}
\def\nil{{\rm nil\,}}

\def\id{{\rm id}}

\newcommand{\secat}{\operatorname{{\text{\rm secat}}}}
\newcommand{\tc}{\operatorname{{\text{\rm TC}}}}
\newcommand{\tcn}{\operatorname{{\text{\rm tc}}}}
\newcommand{\cat}{\operatorname{{\text{\rm cat}}}}

   \title{Topological complexity of the work map}
\author{Aniceto Murillo}
\address{Departamento de Algebra, Geometr\'ia y Topolog\'ia\\
         Universidad de M\'alaga\\
        Ap. 59, 29080 M\'alaga,\\
         Espa\~na}
\email{aniceto@uma.es}

\author{Jie Wu}
\address{Department of Mathematics\\ National University of Singapore\\
 Block S17  Lower Kent Ridge Road \\ 119076 Singapore }
\email{matwuj@nus.edu.sg}

\thanks{This research was partially supported by the Singapore Ministry of Education research grant (AcRF Tier 1 WBS No. R-146-000-222-112), a grant (No.11329101) of NSFC and by the Spanish MINECO research grants MTM2013-41768, MTM2016-78647.  The first author would like to thank the Department of Mathematics of the NUS  for his warm hospitality.}

\begin{document}

\date{\today}

\begin{abstract}
We introduce the topological complexity of the work map associated to a robot system. In broad terms, this measures the complexity of any algorithm controlling, not just the motion of the configuration space of the given system, but the task for which the system has been designed. From a purely topological point of view, this is a homotopy invariant of a map which generalizes the classical topological complexity of a  space.
\end{abstract}

\maketitle

\section*{Introduction}

The theory of topological complexity was initiated by Michael Farber~\cite{far,far2004} and it has become one of most active field in the area of applied topology during last decade. In broad terms, this theory measures the complexity of any algorithm controlling the motion planing on a given configuration space. One can also regard this invariant as the minimum number of navigational instabilities of such a motion planning.

However, in many situations in robotics, more important than controlling the motion on a given configuration space, it  is designing the resulting motion on the corresponding workspace. We briefly support this assertion by  looking at two different key examples:

\smallskip

In robotics, see for instance \cite[Chap.~1]{craig}, a {\em robot manipulator} or {\em robot arm}   consists of multiple rigid segments (\textit{sub-arms}, \textit{links}) where successive, neighboring links are connected by {\em joints} of different kind (rotational, translational, polar, cylindrical).  The \textit{base} of a robot manipulator is the end of the first link  which is fixed to a point through a given joint. The {\em end effector} is the device (screw driver, welding device,...) at the end of the last link of the robotic arm, designed to interact with the environment performing the proposed task.  Finally, The {\em workspace} of a robot arm is defined as the set of points that can be reached by the end effector.

From the topological point of view, the configuration space of a robot arm, i.e., the set of all possible states of such a manipulator, was first modeled in \cite{Gottlieb}, see also \cite{Pfalzgraf} or the modern  reference \cite{far3}.
Less attention has been given to the topological study of the workspace as this  is not an (even diffeomorphic) invariant of the configuration space.
For instance, consider two robot manipulators $C$ and $D$, each of which consisting of two links of lengths $\ell_1,\ell_2$ with $\ell_1>\ell_2$ and both with one degree of freedom. In other words, both $C$ and $D$ are diffeomorphic to $S^1\times S^1$.

\begin{wrapfigure}{r}{0.35\textwidth}
  \vspace{-26pt}
  \begin{center}
    \includegraphics[height=62mm]{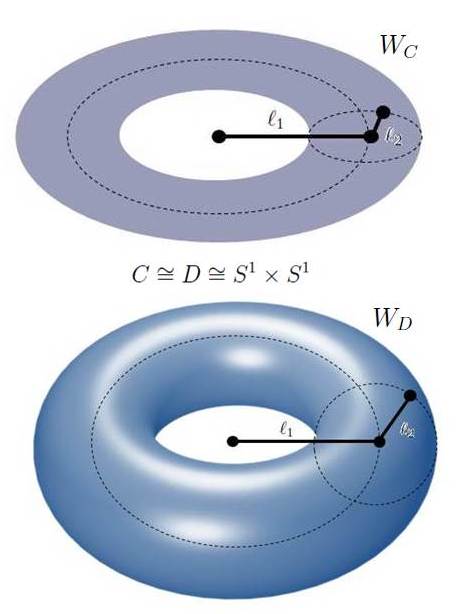}
  \end{center}
  \vspace{-15pt}
\end{wrapfigure}

However, if we assume $C$ to be ``planar'', then its workspace $W_C$ is an annulus of radius $\ell_1-\ell_2$ \cite[\S1.2]{far3}. On the other hand, if in $D$, the circle generated by the end effector link is ``transversal'' to the plane containing the circle generated by the based link, then $W_D$ is a torus. This is discussed in detail in Example 2.6.

Nevertheless, both the configuration space $C$ and the workspace $W$ of a given robot arm are connected by the continuous  {\em work map}
$$
f\colon C\longrightarrow W
$$
which assigns to each state of the configuration space the position of the end effector at that state.

This map is an important object to be considered when  implementing  algorithms controlling the task performed by the robot manipulator. Indeed,
the input of such an algorithm  are  pairs $(a,b)\in W\times W$ of points of the workspace, that is, pairs of possible positions of the end effector. The output for such a pair ought to be a curve in the configuration space $\alpha\in C^I$ such that $f\big(\alpha(0)\bigr)=a$ and $f\big(\alpha(1)\bigr)=b$, where $C^I$ is the space of curves in $C$  (that is the space of continuous maps $I=[0,1]\to C$).

One may argue that a motion planner of the configuration space $C$ produces such an algorithm by composing with the work map. However, the efficiency of such an algorithm might not be optimal as, for instance,  the work map is not injective in general and therefore,
many states of the configuration space may give the same position of the end effector.

\smallskip

The second example to which the above can be applied is the following: according to~\cite{Bajd}, a {\em multi-robot system} consists of two or more robots executing a task requiring collaboration among them. Assume that such a multi-robot system is formed by  $n$ autonomous mobile robots running in a space $X$ without colliding. The configuration space of such a system is the standard $n$-th configuration space
 $$F(X,n)=\{(x_1,\ldots,x_n)\in X^{n} \ | \ x_i\not=x_j\textrm{ for } i\not=j\}$$
  with the subspace topology of the $n$-fold Cartesian product of $X$. On the other hand, the task requiring the collaboration of the robots can be described as a continuous {\em work map}
   $$
   f\colon F(X,n)\longrightarrow Y,
   $$
   depending on  the locations of the robots, and with values on the {\em workspace} $Y$ which is often described as a subspace of some Euclidean space $\mathbb{R}^N$. Hence, an algorithm controlling the task performed by the multi-robot system can be described, as before, in terms of the work map.

\medskip

These examples motivate the main purpose of this article which is to propose the notion of {\em topological complexity of a (work) map} together with its {\em naive}\footnote{The notion of naive topological complexity was suggested by Shmuel Weinberger.} or {\em strict} version as a generalization of the topological complexity of a given configuration space\footnote{Finding an appropriate notion of topological complexity of a map was a problem raised by A. Dranishnikov on his course at the event {\em ``Applied Algebraic Topology, Advanced Courses''} held at the CIEM, Castro Urdiales, Spain, in 2014.}.

 \begin{definition}\label{tcprime} Given a continuous map $f\colon X\to Y$, the {\em topological complexity of $f$}, $\tc(f)$, is the least integer $n\le \infty$ such that $X\times X$ can be covered by $n+1$ open sets $\{U_i\}_{i=0}^n$ on each of which there is a continuous map $\sigma_i\colon U_i\to X^I$ satisfying $$(f\times f)\circ \pi\circ \sigma_i\simeq (f\times f)_{|_{U_i}},$$  where $\pi\colon X^I\to X\times X$ is the path fibration, $\pi(\alpha)=\bigl(\alpha(0),\alpha(1)\bigr)$. The {\em naive} or {\em strict topological complexity of $f$}, $\tcn(f)$ is defined analogously requiring each of the maps $\sigma_i\colon U_i\to X^I$ to satisfy the stronger condition
  $$(f\times f)\circ \pi\circ \sigma_i= (f\times f)_{|_{U_i}}.$$
 \end{definition}

Arising from practical concern, one may want $(f\times f)\circ \pi\circ \sigma_i= (f\times f)_{|_{U_i}}$ rather than homotopic in the above definition. Hence we also introduce the notion of naive or strict topological complexity of a map.

The non strict notion is an appropriate extension of the usual topological complextiy from a homotopy point of view. However, controlling the work map only up to homotopy make the results harder to interpret in terms of actual motion plannings. The  strict version takes care of this gap.

 We point out that our approach is different from the interesting work of P. Pavesic \cite{pa1,pa2} where, roughly speaking, the topological complexity of a map $f\colon X\to Y$ is defined as the Svarc genus of the map $\gamma\colon X^I\to X\times Y$, $\gamma(\alpha)=\bigl(\alpha(0),f(\alpha(1))\bigr)$, which assigns to each curve joining two given states of the configuration space, the initial state and the end effector postion at the final state.

\smallskip

 We first introduce and study in Section 1, for a given pair of maps $E\stackrel{p}{\to}B\stackrel{g}{\to}X$, the $g$-sectional category of $p$, $\secat_g(p)$,  a generalization of the Svarc genus or sectional category. Then, the topological complexity of a map $f\colon X\to Y$ can be thought of as $\secat_{f\times f}(\pi)$. As such, we show in Section 2 that $\tc(f)$ is a an invariant of the homotopy type of $f$ and is $0$ if and only if $f$ is inessential. Moreover, $\tc(X)=\tc(\id_X)$. We also provide upper and lower bounds for the topological complexity of a map, see propositions \ref{cotas} and \ref{cotasco}:
 $$
\max\{\cat(f),\,\nil \ker \cup_{|_\im (f\times f)^*}\} \le \tc(f) \le \min\{\tc(X),\,\cat(f\times f)\}.
$$
 This indicates in particular that the complexity of an algorithm controlling the task performed by a system is in general smaller than the one controlling just the motion of the system. We finish the section with several examples.

 In Section 3 we give a characterization of the topological complexity of the rationalization $f_\bq$ of a given map $f$ between simply connected spaces in terms of its Sullivan models. This is highly computable in algebraic terms and becomes a lower bound of the topological complexity of $f$ as $\tc(f_\bq)\le\tc (f)$. As an application, we show that the topological complexity of a formal map always coincides with its cohomological lower bound.

In Section 4 we present some properties of the naive topological complexity of a map which is in general a rough upper bound of the non naive version and it coincides with it for  fibration. We discuss the particular  case of  an articulated arm whose links have variable length.

  Finally, we would like to stress that our  purpose is not being exhaustive in the study of the the topological complexity of a map, but just laying the groundwork for its further development and presenting the general behaviour of this new invariant.

\section{$f$-Sectional Category}

 In what follows, and unless explicitly stated otherwise, a topological space will always be pointed, path-connected, and of the homotopy type of a CW-complex. Continuous maps are assumed to preserve base points.

We recall the definition of the most classical Lusternik-Schnirelmann invariants. The category of a space $X$, $\cat(X)$, is the least $n\le \infty$ such that $X$ can be covered by $n+1$ open sets contractible within $X$. On the other hand, the category of a map $f$,  $\cat(f)$, is the least $n\le\infty$ such that the domain can be covered by $n+1$ open sets on each of which the restriction of $f$ is homotopically trivial. Finally,
given a map $p\colon E\to B$, the {\em sectional category of $p$} \cite{s} denoted by $\secat(p)$ , is the least $n\le\infty$ for which $B$ can be covered by $n+1$ open sets on each of which there is a local homotopy section of $p$. Here we extend this invariant.

\begin{definition} Let $E\stackrel{p}{\to}B\stackrel{f}{\to}X$ be two continuous maps. An open set $U\subset B$ is {\em $f$-categorical} if there is a map $s\colon U\to E$ such that $fps\simeq f_{|_U}$. We call $s$ an {\em $f$-section}. The {\em $f$-sectional category of $p$}, $\secat_f( p)$, is the least  $n\le \infty$ for which $B$ admits a covering of $n+1$ $f$-categorical open sets.
\end{definition}

Obviously, if $f\simeq g$ and $p\simeq q$, then $\secat_f(p)=\secat_g(q)$. Also, observe that $\secat(p)=\secat_{\id_B}(p)$.

Recall that $p\colon E\to B$ is said to be dominated by $p'\colon E'\to B'$ if there is a (homotopy) commutative diagram,
$$
\xymatrix{E \ar[r]^i\ar[d]_p & E'\ar[r]^r\ar[d]_{p'}&E\ar[d]_p\\
B \ar[r]_j & B'\ar[r]_t&B}
$$
such that $ri\simeq \id_E$ and $tj\simeq \id_B$.

\begin{lemma} Let $p\colon E\to B$ be dominated by $p'\colon E'\to B'$ and let $f\colon B\to X$ be any map. Then, $\secat_{ft}(p')\le \secat_f(p)$.
\end{lemma}

\begin{proof}

 Let $U\subset B$ be $f$-categorical and let $s\colon U\to E$ with $fps\simeq f_{|_U}$. Let $V=t^{-1}(U)$ and $s'=ist_{|_V}\colon V\to E'$. Then,
$ftp's'=ftp'ist\simeq ftjpst\simeq fpst\simeq ft$, that is, $V$ is $ft$-categorical for $p'$.
\end{proof}
We write $p\sim p'$ is there is a homotopy commutative square,
$$
\xymatrix{E \ar[r]^g_\simeq\ar[d]_p & E'\ar[d]^{p'}\\
B \ar[r]_h^\simeq & B'}
$$
in which the horizontal arrows are homotopy equivalences. An immediate consequence of the lemma above is:

\begin{proposition}\label{nodepende} Let $p\sim p'$ and $f\colon B'\to X$. then $\secat_{hf}(p)=\secat_f(p')$.\hfill$\square$
\end{proposition}

The following summarize how the classical bounds for sectional category (see for instance \cite[\S9]{corlupotan}) have to be  modified for this new invariant.

\begin{proposition}\label{primera}
For any $p\colon E\to B$ and any map $f\colon B\to X$,
$$
\secat_f(p)\le \min\{\secat(p),\cat(f)\}.
$$
Moreover, if $E$ is contractible, then $\secat_f(p)=\cat(f)$.
\end{proposition}

\begin{proof} The inequality $\secat_f(p)\le \secat(p)$ is obvious. Also, if $f_{|_U}\simeq *$ the constant map $*\colon U\to B$ is an $f$-section. This proves the other inequality. Finally, if  $U\subset B$ is $f$-categorical and $E$ is contractible, $f_{|_U}$ is homotopically trivial. This shows that $\cat(f)\le\secat_f(p)$ whenever $E$ is contractible, which proves the last assertion.
\end{proof}
In particular, as $\cat(f)\le\min\{\cat(B),\cat(X)\}$, we obtain:
\begin{corollary}
$\secat_f(p)\le \min\{\cat(B),\cat(X)\}$. In particular, if either $B$ or $X$ is a co-$H$-space, then $\secat_f(p)\le 1$.
\end{corollary}

An interesting feature of this invariant is the following.

\begin{proposition} Let $p\colon E\to B$  be the pullback of a fibration $p'\colon E'\to B'$ along $f\colon B\to B'$. If $E'$ is contractible, then,
$$
\secat_f(p)=\cat(f).
$$
\end{proposition}

\begin{proof}
Let $U\subset B$ be an $f$-categorical open set and $s\colon U\to E$ an $f$-section. As $E'$ is contractible then $fps\simeq f{|_U}\simeq *$. This proves that $\cat(f)\le \secat_f(p)$. The other inequality is given by Proposition \ref{primera}
\end{proof}

The following immediate consequences are examples of this situation:

\begin{corollary} Let $p\colon E\to B$ a principal fibration (resp. $G$-bundle) classified by a map $f\colon B\to K(\pi,n)$ (resp. $f\colon B\to BG$). Then,
$$
\secat_f(p)=\cat(f).
$$
\hfill$\square$
\end{corollary}

We now set the lower cohomological lower bound of the $f$-sectional category. Recall that the {\em nilpotency index} of a ring $R$, $\nil R$, is the biggest $n\le\infty$ such that $R^n\not=0$.

\begin{proposition}\label{nilpotencia} $\secat_f(p)\ge \nil{\ker p^*}_{|_{\im f^*}}$.
\end{proposition}
Here $(-)^*$ denotes the morphism induced in reduced cohomology over any fixed ring.

\begin{proof}
Assume $\secat_f(p)=n$ and let $k_i\colon U_i\hookrightarrow B$, $i=1,\dots,n+1$,  be $f$-categorical open sets covering $B$ with $f$-sections $s_i$. Consider the long exact sequence
$$
\cdots\to H^*(B,U_i)\stackrel{q_i^*}{\to} H^*(B)\stackrel{k_i^*}{\to}H^*(U_i)\to\cdots
$$
induced by the pair $(B,U_i)$ and let $\gamma_1=f^*(\alpha_1),\dots,\gamma_{n+1}=f^*(\alpha_{n+1})\in\ker p^*\cap\im f^*$. Then,
$$
k_i^*(\gamma_i)=k_i^*f^*(\alpha_i)=k_i^*s_i^*p^*f^*(\alpha_i)=0,
$$
and thus $\gamma_i\in \ker k_i^*=\im q_i^*$. Write $\gamma_i=q_i^*(\overline\gamma_i)$, $\overline\gamma_i\in H^*(B,U_i)$. To finish observe that $\overline\gamma_1
\cup\dots\cup\overline\gamma_{n+1}\in H^*(B,B)=0$ and, denoting $q\colon B\to (B,B)$, we have
$
\gamma_1\cup\dots\cup\gamma_{n+1}=q^*(\overline\gamma_1
\cup\dots\cup\overline\gamma_{n+1}
)=0$.

\end{proof}

Observe that $\nil{\ker p^*}_{|_{\im f^*}}$ is in general smaller than $ \nil \ker (fp)^*$ which is the classical cohomological lower bound of $\secat(fp)$.

Next we give the ``Ganea and Whitehead characterizations'' of the $f$-sectional category. For the first, we follow the classical approach of \cite{ja}, improved in \cite[\S2]{ferghienkahlvan}, with the suitable modifications.

Recall that given a fibration $q\colon Z\to Y$, the {\em $n$-fold join of $q$} is the space $*^n_YZ$ inductively defined as follows: $*^0_YZ=Z$; $*^1_YZ=Z*_YZ$ is the double mapping cylinder of the projections of $Z\times_YZ$ over $Z$,
$$
Z*_YZ=\bigl((Z\times_YZ)\times I\amalg Z\amalg Z\bigr)/(x,y,0)\sim x,(x,y,1)\sim y.
$$
Finally, $*^n_YZ=(*^{n-1}_YZ)*_YZ$. The {\em $n$-fold join fibration} is the fibration $*^n_Yq\colon *^n_YZ\to Y$ inductively defined by $*^0_Yq=q$; $*^1_Yq=q*_Yq$ where $(q*_Yq)[x,y,t]=q(x)=q(y)$; and $*^n_Yq=(*^{n-1}_Yq)*_Yq$.

Now, factor a given composition $E\stackrel{p}{\to}B\stackrel{f}{\to}Y$ as $qj$ where
$j\colon E\stackrel{\simeq}{\to} Z$ is a homotopy equivalence and $q\colon Z\to Y$ is a fibration. Then, we have:

\begin{proposition}\label{ganea} $\secat_f(p)$ is the least integer $n$ for which there exists $\sigma\colon B \to Z$ such that $(*^n_Yq)\sigma=f$.
\end{proposition}

\begin{proof}
Assume $\secat_f(p)=n$. By induction on $m$, with $0\le m\le n$ we show the existence of an open covering $\{U_i\}_{i=0}^{n-m}$ of $B$ for which:

There exists a map $\sigma_0\colon U\to *^m_YE$ such that $(*^m_Yq)\sigma_0=f_{|_{U_0}}$.

There are maps $ \sigma_i\colon U_i\to Z$, for $i=1,\dots,n-m$, such that $q\sigma_i=f_{|_{U_i}}$.

For $m=0$ choose a covering $\{U_i\}_{i=0}^{n}$ of $B$ and $f$-sections $\tau_i\colon U_i \to E$ of $p$ . As $q$ is a fibration we may replace the maps $j\tau_i$ by $\sigma_i\colon U_i\to Z$ so that $q\sigma_i=f_{|_{U_i}}$.

Let $m<n$ and $\mathcal U=\{U_i\}_{i=0}^{n-m}$ an open covering  of $B$ with maps $\sigma_0\colon U\to *^m_YE$, $ \sigma_i\colon U_i\to Z$, for $i=1,\dots,n-m$, satisfying the induction hypothesis. Choose refinements of $\mathcal U$, $\{V_i\}_{i=0}^{n-m}$, $\{W_i\}_{i=0}^{n-m}$ with
$$
V_i\subset \overline V_i\subset W_i\subset\overline W_i\subset U_i,
$$
and consider the disjoint closed subspaces $A_0=\overline V_0\cap(B-W_1)$, $A_1=\overline V_1\cap(B-W_0)$. Observe that $A_0\cap A_1\cap C=\overline V_0\cup \overline V_1$ being $C=\overline W_0\cap\overline W_1\cap(\overline V_0\cup \overline V_1)$. By the Urysohn Lemma choose a continuous function $h\colon B\to I$ such that $h(A_0)=0$ and $h(A_1)=1$ and define $\sigma\colon \overline V_0\cup \overline V_1\to *^{m+1}_YZ$,
$$
\sigma(x)=\begin{cases} [\sigma_0(x)],& x\in A_0,\\ [\sigma_1(x)],& x\in A_1,\\
[\sigma_0(x),\sigma_1(x),h(x)],& x\in C.\end{cases}
$$
To finish, consider the open covering of $B$ given by $\{V_0\cap V_1, U_i\}_{i=2}^{n-m}$ and the maps $\sigma\colon V_0\cup V_1\to *^{m+1}_YZ$, $\sigma_i\colon U_i\to Z$, $i=2,\dots,m-n$.

Conversely, assume that  $\sigma\colon B \to Z$ is such that $(*^n_Yq)\sigma=f$. By \cite[Lemma 2.4]{ferghienkahlvan} the sectional category of the fibration $q_n$ given by  the pullback diagram
$$
\xymatrix{(*^n_YZ)\times_YZ \ar[r]^(0.65){\pi}\ar[d]_{q_n} & Z\ar[d]^q\\
{*^n_YZ} \ar[r]^{*^n_Yq}& Y}
$$
is bounded above by $n$. Hence we may choose an open covering $\{V_i\}_{i=0}^n$ of $*^n_YZ$ and local sections $\gamma_i\colon V_i\to (*^n_YZ)\times_YZ$ of $q_n$. For each $i=0,\dots,n$ consider $U_i=\sigma^{-1}(V_i)$ and $\tau_i=\pi\gamma_i\sigma\colon U_i\to Z$. Then $\{U_i\}_{i=0}^n$ is an open covering of $B$ and the composition $\sigma_i=k\tau_i\colon U_i\to E$, with $k$ a homotopy inverse of $j$, is a local $f$-section of $p$.
\end{proof}

As $\secat_f(p)$ is a homotopy invariant it is useful to have a less rigid analogue of the above result. For it, and abusing of notation, we denote by $E*_YE$ the (homotopy) join, i.e., the homotopy pushout of the homotopy pullback of $E\stackrel{fp}{\to} Y\stackrel{fp}{\leftarrow} E$. More generally, define $*^0_YE=E$ and $*^n_YE=(*^{n-1}_YE)*_YE$.  By the weak universal property of the homotopy pushout we also get maps, $*^n_Yfp\colon *^n_YE\to Y$. Observe that $*^n_YE\simeq *^n_YZ$ and the Proposition above readily implies:

\begin{theorem}\label{ganeahomo} $\secat_f(p)$ is the least integer $n$ for which there exists $\sigma\colon B \to E$ such that $(*^n_Yfp)\sigma\simeq f$.\hfill$\square$
\end{theorem}

\begin{remark}\label{diagrama} Note that for each $n\ge 0$,  there are maps $\iota_n\colon E\to *^n_YE$ inductively defined by the construction of the homotopy join, making commutative the following diagram,
$$
\xymatrix{E \ar[r]^(0.45){\iota_n}\ar[d]_{p} & {*^n_YE}\ar[d]^{*^n_Yfp}\\
B \ar[r]^{f}& {Y}}
$$
\end{remark}

For the Whitehead characterization of $\secat_f(p)$ factor $fp$ as $qj$ where $j\colon E\hookrightarrow Z$ is a cofibration and $q\colon Z\stackrel{\simeq}{\to} Y$ is a homotopy equivalence. Recall that the {\em $n$th fat wedge of $j$} is defined as
$$
T^n(j)=\{(x_0,\dots,x_n)\in Z^{ n+1}\,\,\text{such that $x_k\in\im j$ for some $k$}\}.
$$

Consider the composition
\begin{equation}\label{hn}
h_n\colon T^n(j)\hookrightarrow Z^{ n+1}\stackrel{\simeq}{\to} Y^{ n+1}
\end{equation}
  and observe  \cite[Thm.~3.3.2]{fa}, \cite{ma} that there is a homotopy pullback of the form
$$
\xymatrix{{*^n_YE} \ar[r]^{\eta}\ar[d]_{*^n_Yfp} & {T^n(j)}\ar[d]^{h_n}\\
Y \ar[r]^(0.40){\Delta^n}& {Y^{n+1}}}
$$
where $\Delta^n$ is  the $n$-diagonal map. Then:

\begin{theorem}\label{whitehead} $\secat_f(p)$ is the least integer $n$ for which the map $\Delta^n f\colon B\to Y^{\times n+1}$ homotopy factors through the $n$th fat wedge $T^n(j)$,
$$
\xymatrix{ & {T^n(j)}\ar[d]^{h_n}\\
B\ar[r]_(.37){\Delta^n f}\ar@{-->}[ru]& {Y^{ n+1}}.}
$$
\end{theorem}
\begin{proof} If $\secat_f(p)=n$ then, via Theorem \ref{ganeahomo}, there exists a map $\sigma\colon B\to *^n_YE$ such that $(*^n_Yfp)\sigma\simeq f$. Hence, the map $\eta\sigma\colon B\to T^n(j)$ is the dotted lifting in the above diagram.

Conversely, given $\xi\colon B\to T^n(j)$ such that $h_n\xi\simeq f\Delta$, the weak universal property of the homotopy pullback produces a map $\sigma\colon B\to  *^n_YE$ such that $\eta\sigma\simeq\xi$ and $(*^n_Yfp)\sigma\simeq f$.
\end{proof}

\section{Topological complexity of a map}\label{section-on-target-topological-complexity}

The following is Definition \ref{tcprime} in terms of sectional category.

\begin{definition} Let $f\colon X\to Y$ be a continuous map. The {\em topological complexity of $f$} is defined as
$$
\tc(f)=\secat_{f\times f}(p)
$$
where $\pi\colon X^I\to X\times X$ is the path fibration, $\pi(\alpha)=\bigl(\alpha(0),\alpha(1)\bigr)$. In view of the commutative diagram
$$\xymatrix{
 X\ar[r]^(0.43)\simeq_(0.43)c\ar[d]_(0.43)\Delta
 &
 X^I\ar[ld]^(0.43){\pi}\\
 X\times X
 &
  }$$
  where $c(x)$ is the constant path in $x$, Proposition \ref{nodepende} implies that
  $$
  \tc(f)=\secat_{f\times f}(\Delta).
$$
\end{definition}
Note that
$
\tc(X)=\tc(\id_X)$.

\begin{proposition}\label{cotas} For any map $f\colon X\to Y$,
$$
\cat(f)\le \tc(f)\le \min\{\tc(X),\cat(f\times f)\}.
$$
\end{proposition}
\begin{proof}
Proposition \ref{primera} proves the second inequality. For the first, fix $x_0\in X$,  suppose $s\colon U\to X^I$ is and $f\times f$-categorical section of $\pi$, and let $V=\{x\in X,\,(x_0,x)\in U\}$. Consider the map $k\colon V\to U$, $k(x)=(x_0,x)$ and denote by $q_2\colon X\times X\to X$ the projection over the second factor. Then,
$$
q_2(f\times f)\pi sk\simeq f_{|_V}\colon V\to Y.
$$
Hence, in view of the commutative diagram
 $$\xymatrix{
 X^I\ar[r]^{f^I}\ar[d]_(0.43)\pi& Y^I\ar[d]_(0.43)\pi\\
 X\times X\ar[r]_{f\times f}& Y\times Y,\\
  }$$
  it follows that
  $$
  q_2(f\times f)\pi sk=q_2\pi f^Isk.
  $$
  However, observe that the image of the map $f^Isk$ lies in the contractible subspace of $Y^I$ of paths starting at $f(x_0)$. Thus,
  $$
  f_{|_V}\simeq q_2(f\times f)\pi sk\simeq *.
  $$
  Finally, observe that if $\{U_i\}_{i\in J}$ is an open covering with each $U_i$ as above, then the corresponding family $\{V_i\}_{i\in J}$ is an open covering of $X$. This proves the first inequality.
  \end{proof}
  Taking into account that in general, $\cat(f\times g)\le \cat(f)+\cat(g)$, we immediately deduce:
  \begin{corollary}\label{trivial}
  $\tc(f)=0$ if and only if $f\simeq *$.\hfill$\square$
  \end{corollary}

  We also prove that for maps $f$ of zero sectional category, its topological complexity coincides with the topological complexity of the codomain:

\begin{proposition}\label{propositionsecat}
 Let $f\colon X\to Y$ be a continuous map admitting a homotopy cross section $s\colon Y\to X$.  Then ${\tc}(f)=\tc(Y)$.
 \end{proposition}
 \begin{proof}
 Suppose that $\tc(Y)=n$, that is, there exists an open cover $\{V_0,V_1,\ldots,V_n\}$ of $Y\times Y$ with maps $\theta_i\colon V_i\to Y^I$ that are homotopy  liftings of the inclusion map $V_i\subseteq Y\times Y$. Let $U_i=(f\times f)^{-1}(V_i)\subseteq X\times X$ and let $\sigma_i\colon U_i\to X^I$ be the composite
 $$
 \sigma_i\colon U_i\stackrel{f\times f|_{U_i}}{\longrightarrow} V_i\stackrel{\theta_i}{\longrightarrow} Y^I\stackrel{s^I}{\longrightarrow} X^I.
 $$
 Then,
 \begin{equation}\label{compu}
 \begin{array}{rcl}
(f\times f)\circ \pi\circ\sigma_i&=&\pi\circ f^I\circ \sigma_i\\&=&\pi\circ f^I\circ s^I\circ \theta_i\circ (f\times f|_{U_i})\\
 &\simeq& \pi\circ \theta_i\circ (f\times f|_{U_i})\\
 &\simeq&f\times f|_{U_i},\\
 \end{array}
\end{equation}
and therefore, ${\tc}(f)\leq \tc(Y)$.

 Conversely suppose that ${\tcn}(f)=n$, that is, there exists an open cover $\{U_0,\ldots,U_n\}$ of $X\times X$ and maps $\sigma_i\colon U_i\to X^I$ such that $(f\times f)\circ \pi\circ \sigma_i\simeq (f\times f)|_{U_i}$. Let $V_i=(s\times s)^{-1}(U_i)\subseteq Y\times Y$ and let
 $$\theta_i=f^I\circ \sigma_i\circ (s\times s|_{V_i})\colon V_i\longrightarrow Y^I.$$ Then $\theta_i$ is a homotopy lifting of the inclusion $V_i\subseteq Y\times Y$. Thus, $\tc(Y)\leq {\tc}(f)$, and hence the result.

 \end{proof}

 This result suggests another upper bound for the topological complexity of $f\colon X\to Y$: Let $n$ be the smallest integer for which $Y\times Y$ admits an open cover $V_0,\dots, V_n$ such that, for each $i$, there are homotopy local sections $\tau_i\colon V_i\to V_i^I$ and $s_i\colon V_i\to X\times X$ of $\pi$ and $f\times f$ respectively. Obviously $\tc(Y)\le n$.

\begin{proposition}\label{anotherupper} $\tc(f)\le n$.
\end{proposition}

\begin{proof}
 For $0\leq i\leq n$ let $U_i=(f\times f)^{-1}(V_i)$ and define
 $$\sigma_i=s_i^I\circ\tau_i\circ (f\times f)\colon U_i\longrightarrow X^I.$$
 An analogous computation to the one in (\ref{compu}) shows that
$ (f\times f)\circ \pi \circ \sigma_i\simeq (f\times f)|_{U_i}$.
\end{proof}
On the other hand, Proposition  \ref{nilpotencia} readily implies:

\begin{proposition}\label{cotasco} For any map $f\colon X\to Y$,
$$
\nil \ker \cup_{|_{\im (f\times f)^*}} \le \tc(f).
$$
\hfill$\square$
\end{proposition}

\begin{example}\label{ejemplo1}
(1) Obviously  $\tc(f)=0$ for any $f\colon S^n\to S^m$, $n<m$. For $f\colon S^n\to S^n$, $\tc(f)$ is also zero if deg$f=0$. If deg$f>0$, then
$$
\tc(f)=\tc(S^n)=\begin{cases} 1,&\text{if $n$ is odd,}\\
2, &\text{if $n$ is even,}\end{cases}
$$
Indeed, by Proposition \ref{cotas}, if $n$ is odd, $\tc(f)\le \tc(S^n)=1$. On the other hand, by Corollary \ref{trivial},  $\tc(f)$ cannot be zero as $f$ is essential.

If $n$ es even, again by Proposition \ref{cotas}, $\tc(f)\le \tc(S^n)=2$. On the other hand, choose a non zero class $\alpha\in H^n(S^n)$ which is also in the image of $f^*$ and observe that $\gamma=\alpha\otimes1-1\otimes\alpha\in \ker\cup$ and $\gamma^2\not=0$. Hence, by Proposition \ref{cotasco}, $\tc(f)\ge 2$.

(2) Let $f\colon \bc P^n\to \bc P^m$ with $1\le n\le m$. These maps are also classified by the degree. If deg $f=0$, then $\tc(f)=0$. If deg$f>0$, then choose again  a non zero class $\alpha\in H^2(\bc P^n)$ which is also in the image of $f^*$ and observe that $\gamma=\alpha\otimes1-1\otimes\alpha\in \ker\cup$ and $\gamma^{2n}\not=0$. Hence, by Proposition \ref{cotasco}, $\tc(f)\ge 2n$. On the other hand, by Proposition \ref{cotas}, $\tc(f)\le \tc(\bc P^n)=2n$. Hence, $\tc(f)= 2n$.

(3) For any map $f\colon X\to Y$ into a co-$H$-space, $\tc(f)\le 2$. Indeed, co-$H$-spaces have category $1$ and therefore, by Proposition \ref{cotas}, $\tc(f)\le \cat(f\times f)\le 2\cat f \le 2\cat Y=2$.

\end{example}

\noindent
\begin{example} ($2$-link robot arm).
Let $T=S^1\times S^1$ be the configuration
space of a robot  manipulator  of two links of lengths $\ell_1>\ell_2$  and both with one degree of freedom.
 Then, the work map may have two different homotopy types:

 \begin{center}
    \includegraphics[height=54mm]{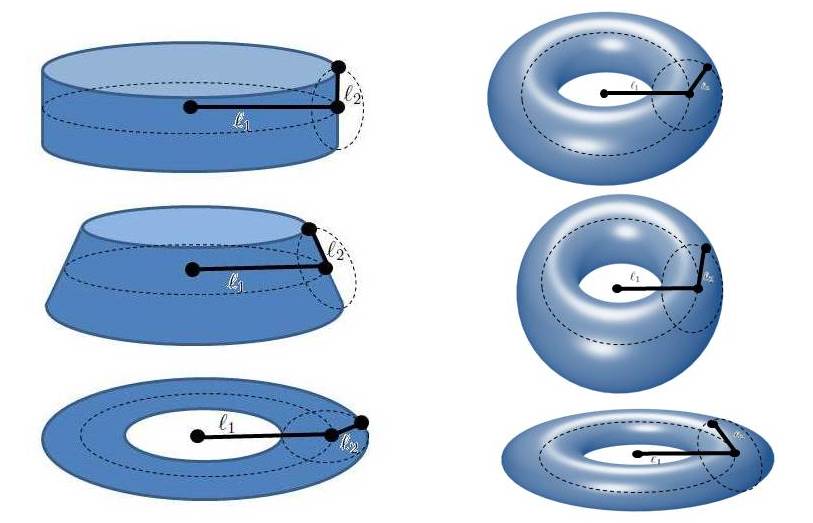}
  \end{center}

Assume that, at each state of the configuration space, the plane defined by the circle generated by the end effector link contains the tangent vector of the circle generated by the base link.
Then, the workspace  is a strip homotopy equivalent to $S^1$ and the work map  is homotopic to the projection $f\colon T\longrightarrow S^1$ over the first factor.

We show that in this case $\tc(f)=1$. Indeed, as $\tc(S^1)=1$ let $\sigma_i\colon V_i\to (S^1)^I$ be local sections of the path fibration $\pi$, $i=1,2$, with $V_1\cup V_2=S^1$. Let $U_i=f^{-1}(V_i)$ and choose a section $s\colon S^1\to C$ of the projection $f$. Define
$
\alpha_i\colon U_i\to T^I$, $\alpha_i=s^I\sigma_i(f\times f)$, for $i=1,2$.
Obviously $(f\times f)\pi\alpha_i\simeq (f\times f)_{|_{U_i}}$.

In the remaining cases, the workspace is a torus and the work map  is homotopic to the identity $\id_{T}$. Hence $\tc(f)=\tc(S^1\times S^1)=2$.

\end{example}

\section{Rational topological complexity of a map}
As rational homotopy theory is particularly fond of algorithms which permit  explicit computations we give in this section a characterization of the  topological complexity of a map in the rational homotopy category. For it, we will be using known results in rational homotopy  for
which the excellent reference \cite{fehatho} is now standard. Here
we simply present a brief summary of some basic facts.  Any space $X$ considered within this section is simply connected and of the homotopy type of a
 CW-complex of finite type. Its rationalization $X_{\Q}$ is
a rational space (its homotopy groups are rational
vector spaces) together with a  map $X\to X_{\Q}$ inducing
isomorphisms in rational homotopy.  On the other hand, to any
space $X$  there corresponds, in a contravariant way, a rational commutative differential graded algebra (cdga henceforth) of the form,
$(\Lambda V,d)$  which  algebraically models the rational
homotopy type of the space $X$, or equivalently, the homotopy
type of its rationalization $X_{\Q}$. By
$\Lambda V$ we mean the free commutative algebra generated by
the graded vector space $V=\oplus_{p\ge 2}V^p$, i.e., $\Lambda V=TV/I$ where
$TV$ denotes the tensor algebra over $V$ and $I$ is the ideal
generated by $v\otimes w-(-1)^{|w||v|}w\otimes v$, $\forall
v,w\in V$, homogeneous elements of degrees $|v|$ and $|w|$
respectively. Moreover, there exists a well ordered basis of $V$ consisting of homogeneous elements $\{v_i\}_{i\in I}$ such that, for each $i$, the differential $dv_i$ is a polynomial on the generators $\{v_j\}_{j<i}$.  The cdga $(\Lambda V,d)$, or simply $\Lambda V$ when the differential is implicitly considered, is called a {\em Sullivan model} of $X$. In general, a {\em model} of $X$ is any cdga connected by quasi-isomorphisms to a Sullivan model of $X$.
This correspondence yields an equivalence between the homotopy
categories of simply connected  rational spaces of finite type and
that of   cdga's  of finite type and also simply connected.

For the rest of the section fix a map $f\colon X\to Y$ and let $\psi\colon (\Lambda V,d)\twoheadrightarrow(\Lambda W,d)$ be a surjective Sullivan model of its rationalization $f_\bq$. denote by $K$ the kernel of the composition
$$
\Lambda V\otimes \Lambda V\stackrel{\psi\otimes \psi}{\longrightarrow}\Lambda W\otimes\Lambda W \stackrel{\mu}{\longrightarrow}\Lambda W
$$
where $\mu$ denotes the multiplication. Then, we prove:

\begin{theorem}\label{racional} $\tc(f_\bq)$ is the least $n$ for which $\psi\otimes\psi$ factors up to homotopy through $(\Lambda V\otimes\Lambda V)/K^{n+1}$.
\end{theorem}

In other words, $\tc(f_\bq)$ is the least $n$ for which, decomposing the projection $\Lambda V\otimes\Lambda V\to(\Lambda V\otimes\Lambda V)/K^{n+1}$ as the composition
$$
\xymatrix{ \Lambda V\otimes\Lambda V\ar[rd]\ar@{>->}[r]&\Lambda V\otimes\Lambda V\otimes\Lambda U\ar[d]^\simeq\\
&(\Lambda V\otimes\Lambda V)/K^{n+1}}
$$
of a cdga cofibration (i.e., a Sullivan extension) and a quasi-isomorphism, there exists a cdga morphism $\rho\colon \Lambda V\otimes\Lambda V\otimes\Lambda U\to\Lambda W\otimes \Lambda W$ such that, the following diagram commutes,
$$
\xymatrix{ \Lambda V\otimes\Lambda V\ar@{>->}[rd]\ar[r]^{\psi\otimes\psi}&\Lambda W\otimes\Lambda W\\
&\Lambda V\otimes\Lambda V\otimes\Lambda U\ar[u]_{\rho}.}
$$

\begin{remark}  (i) Standard arguments on classical localization, taking into account that rationalization commutes with homotopy fibration and cofibration sequences, let us assert that
$$
\tc(f_\bq)\le \tc(f).
$$
Hence,  Theorem \ref{racional}  produces an algebraic lower bound for the target topological complexity.

(ii) Observe that in the case $f=\id_X$, we have $\psi=\id_{\Lambda V}$ and Theorem \ref{racional} recovers the characterization of $\tc(X_\bq)$ in \cite[Thm.~2]{ca}.

(iii) For completeness we recall how  to obtain a surjective model of a given map. Let $ A\to B$ any cdga model of $f_\bq$. Choose a surjective cdga morphism $\alpha\colon \Lambda R\twoheadrightarrow B$ and extend the model above to  $\gamma\colon  A\otimes \Lambda R\otimes \Lambda dR\twoheadrightarrow B$ by $\gamma(w)=\alpha(w)$, $\gamma(dw)=d\alpha(w)$, $w\in R$. Then, the inclusion $A\stackrel{\simeq}{\hookrightarrow} A\otimes \Lambda R\otimes \Lambda dR$ is a quasi-isomorphism and thus, $\gamma$ is a surjective model of $f_\bq$.
\end{remark}
 In the proof of Theorem \ref{racional} we will use the following results:

 On the one hand, recall that
 $
 \tc(f)=\secat_{f\times f}(\Delta)$ and let $$h_n\colon T^n(j)\to (Y\times Y)^{ n+1}$$ be the associated map (\ref{hn}) in this particular case. Then:

 \begin{lemma}\label{lema1}{\em \cite[Thm.~8.4.1]{fa}, \cite[Thm.~1]{fetan}} A model of $h_n$ is given by the projection
 $$
 (\Lambda V\otimes\Lambda V)^{\otimes n+1}\longrightarrow (\Lambda V\otimes\Lambda V)^{\otimes n+1}/K^{\otimes n+1}.
 $$
 \hfill$\square$
 \end{lemma}
On the other hand, consider the diagram of Remark \ref{diagrama} in our particular case,
\begin{equation}\label{paramodelo}
\xymatrix{X \ar[r]^(0.45){\iota_n}\ar[d]_{\Delta} & {*^n_{Y\times Y}X}\ar[d]^{*^n_Y(f\times f)\Delta}\\
X\times X \ar[r]^{f\times f}& {Y\times Y.}}
\end{equation}
Then, a straightforward adaptation to our context of the proof of \cite[Prop.~7]{ca}, based mainly on \cite[Lemma 5]{ca}, gives the following.

\begin{lemma}\label{lema2} There is a model of diagram (\ref{paramodelo})
$$
\xymatrix{{\Lambda V\otimes\Lambda V} \ar[r]^(0.65){k_n}\ar[d]_{\psi\otimes\psi} & {C_n}\ar[d]^{q_n}\\
{\Lambda W\otimes\Lambda W} \ar[r]^(0.60){\mu}& {\Lambda W}}
$$
for which $(\ker q_n)^{n+1}=0$.\hfill$\square$
\end{lemma}

\begin{proof}[Proof of Theorem \ref{racional}]

 Consider the commutative diagram induced by the $n$th multiplication $\mu_n$,
 $$
 \xymatrix{{(\Lambda V\otimes\Lambda V)^{\otimes n+1}} \ar[r]^(0.55){\mu_n}\ar[d] & {\Lambda V\otimes\Lambda V}\ar[d]\\
{ (\Lambda V\otimes\Lambda V)^{\otimes n+1}/K^{\otimes n+1}} \ar[r]^(0.55){\overline\mu_n}& {(\Lambda V\otimes \Lambda V)/K^{n+1},}}
$$
 and assume that $\psi\otimes\psi\colon \Lambda V\otimes \Lambda V\to\Lambda W\otimes\Lambda W$ factors up to homotopy through $(\Lambda V\otimes \Lambda V)/K^{n+1}$, for some $n$.

 Then, the composition,
 $$
 (\Lambda V\otimes \Lambda V)^{\otimes n+1}\stackrel{\mu_n}{\longrightarrow}{\Lambda V\otimes \Lambda V}\stackrel{\psi\otimes\psi}{\longrightarrow}{\Lambda W\otimes \Lambda W}
 $$
 factors up to homotopy through the projection
 $$
 (\Lambda V\otimes\Lambda V)^{\otimes n+1}\longrightarrow (\Lambda V\otimes\Lambda V)^{\otimes n+1}/K^{\otimes n+1}.
 $$
But by Theorem \ref{whitehead} and Lemma \ref{lema1}, $\tc(f_\bq)$ is precisely the least $n$ for which this occurs. Hence $\tc(f_\bq)\le n$.

Conversely, assume that $\tc(f_\bq)=n$ and observe that, by Theorem \ref{ganeahomo} and considering the model in Lemma \ref{lema2}, $\tc(f_\bq)$ is precisely the least $n$ for which $\psi\otimes\psi$ factors up to homotopy through $k_n$. But $k_n(K^{n+1})\subset (\ker q_n)^{n+1}=0$. Hence, $k_n$ factors through $(\Lambda V\otimes\Lambda V)/K^{n+1}$ and therefore, $\psi\otimes\psi$ factors up to homotopy through $(\Lambda V\otimes\Lambda V)/K^{n+1}$.
\end{proof}

As an application we show that the topological complexity of certain rational work maps are given in purely cohomological terms.

Recall that a map $f\colon X\to Y$ is formal if there is a  commutative diagram
\begin{equation}\label{diagrama2}\xymatrix{
 (\Lambda V,d)\ar[r]^{\psi}\ar[d]_(0.43)\alpha^(0,43)\simeq& (\Lambda W,d)\ar[d]_(0.43)\simeq^(0,43)\beta\\
 H^*(\Lambda V,d)\ar[r]_{\psi^*}& H^*(\Lambda W,d),\\
  }
  \end{equation}
in which $\psi$ is a model of $f$. Note that, in particular, $X$ and $Y$ are formal spaces, that is, their rational homotopy type depends only on its rational cohomology. We show that for such maps, the cohomological bound of Proposition \ref{cotasco} is always reached. This generalizes \cite[Thm.~1.2]{lemu}, see also \cite[Cor.~2.2]{jemupa}.

\begin{theorem} For any formal map $f$,
$$
\tc(f_\bq)= \nil \ker \cup_{|_{\im (f\times f)^*}} .
$$
\end{theorem}
Here, cohomology is considered with rational coefficients.
\begin{proof} In view of Proposition \ref{cotasco} it is enough to show that $$ \tc(f_\bq)\le \nil \ker \cup_{|_{\im (f\times f)^*}}.$$
The diagram (\ref{diagrama2}),  in which we may assume that  $\psi$ is a surjective model of $f$, produces a commutative diagram,
$$
\xymatrix{
 \Lambda V\otimes \Lambda V\ar[r]^{\psi\otimes\psi}\ar[d]_(0.43)\simeq^(0,43){\alpha\otimes\alpha}& \Lambda W \otimes\Lambda W\ar[r]^(0.56){\mu}\ar[d]_(0.43)\simeq^(0,43){\beta\otimes \beta}&\Lambda W\ar[d]_(0.43)\simeq^(0,43)\beta\\
 H^*(\Lambda V)\otimes H^*(\Lambda V)\ar[r]_{\psi^*\otimes\psi^*}& H^*(\Lambda W)\otimes H^*(\Lambda W)\ar[r]_(0.63){\mu^*}& H^*(\Lambda W).\\
  }$$
This induces, for any $n\ge 1$, another commutative diagram
$$\xymatrix{
 \Lambda V\otimes\Lambda V\ar[r]^(0,40){(\beta\otimes\beta)(\psi\otimes\psi)}\ar[d]& H^*(\Lambda W)\otimes H^*(\Lambda W)\ar[d]\\
 (\Lambda V\otimes\Lambda V)/K^n\ar[r]& \bigl(H^*(\Lambda W)\otimes H^*(\Lambda W)\bigr)/L^n\\
  }
  $$
  where
 $K=\ker\mu(\psi\otimes\psi)$ and $L=\ker {\mu^*}_{|_{\im \psi^*\otimes \psi^*}}$, which is in turn identified with $\ker \cup_{|_{\im (f\times f)^*}}$.
 Now, assume $\nil \ker \cup_{|_{\im (f\times f)^*}}=n$. Hence $L^{n+1}=0$ and, for $n+1$, the above diagram becomes
 $$\xymatrix{
 \Lambda V\otimes\Lambda V\ar[r]^(0,40){(\beta\otimes\beta)(\psi\otimes\psi)}\ar[d]& H^*(\Lambda W)\otimes H^*(\Lambda W)\\
 (\Lambda V\otimes\Lambda V)/K^{n+1}\ar[ru]&\\
  }
  $$
 As $\beta\otimes\beta$ is a quasi-isomorphism, this shows that $\psi\otimes\psi$ factors up to homotopy through $(\Lambda V\otimes\Lambda V)/K^{n+1}$. By Theorem \ref{racional} the result follows.

\end{proof}

\begin{example}  Let $f$ be any of the maps considered in Example \ref{ejemplo1}. Such a map is formal and therefore,
$$
\tc(f_\bq)=\tc(f)=\nil \ker \cup_{|_{\im (f\times f)^*}}.
$$
\end{example}

\section{Naive or strict topological complexity of a map}
As we remarked in the Introduction, arising from practical consideration in robotics, it is of interest to consider the naive or strict version of the topological complexity of the work map which we recall here:

 \begin{definition}\label{definition4.1} Given a continuous map $f\colon X\to Y$, the {\em naive topological complexity of $f$}, ${\tcn}(f)$, is the least integer $n\le \infty$ such that $X\times X$ can be covered by $n+1$ open sets $\{U_i\}_{i=0}^n$ on each of which there is a continuous map $\sigma_i\colon U_i\to X^I$ satisfying
$$
 (f\times f)\circ \pi\circ \sigma_i= (f\times f)|_{U_i},
$$
 \end{definition}

 \begin{proposition}\label{proposition4.3}
 Let $f\colon X\to Y$ be a continuous map. Then $\tc(f)\leq {\tcn}(f)$ and they coincide if $f$ is a fibration.
 \end{proposition}

 \begin{proof} The first assertion follows from the definition. For the second simply remark that $\pi$ is a fibration and $f\times f$ is also a fibration whenever $f$ is. In this case  $(f\times f)|_{U_i}$ has a lifting to $(f\times f)\circ \pi$ if and only if it has a homotopy lifting.
\end{proof}

Notice also the following property:

\begin{proposition} For any map $f\colon X\to Y$, $\tcn(f)\le \tc(X)$ and equality holds if $f$ is injective.
\end{proposition}

\begin{proof} The first assertion is trivial. The second is also easy: let $U\subset X\times X$ and let $\sigma\colon U\to X^I$ such that  $(f\times f)\circ \pi\circ \sigma= (f\times f)|_{U}$. Since $f$ is injective,  $\pi\circ \sigma=\id_{U}$ and thus $\tc(X)\le \tcn(f)$.
\end{proof}

\begin{remark} In general,  the gap between $\tcn(f)$ and $\tc(f)$ can be arbitrarily large. Choose the inclusion $k\colon X\hookrightarrow CX$ of a given space into its cone. By Proposition above $\tcn(k)=\tc(X)$ while, in view of Corollary \ref{trivial}, $\tc(k)=0$.
\end{remark}

Observe that the same proof of Proposition \ref{propositionsecat}, replacing homotopy liftings by strict ones, proves the following:
\begin{proposition}\label{proposition4.4}
 Let $f\colon X\to Y$ be a continuous map admitting a section $s\colon Y\to X$.  Then ${\tcn}(f)=\tc(Y)$.\hfill$\square$.
 \end{proposition}

 The strict version of Proposition \ref{anotherupper} reads as follows: Let $n$ be the smallest integer for which $Y\times Y$ admits an open cover $V_0,\dots, V_n$ such that, for each $i$, there are  local sections $\tau_i\colon V_i\to V_i^I$ and $s_i\colon V_i\to X\times X$ of $\pi$ and $f\times f$ respectively.

\begin{proposition}\label{anotheruppern} $\tcn(f)\le n$.\hfill$\square$
\end{proposition}

If $f$ fails to be a fibration, the situation becomes more complicated and further analysis may have to be taken in each particular case. For its particular interest we discuss here the motion planning of a robot arm in $\mathbb{R}^3$, with $k$ joints and where each of the links or subarms may have variable length. We assume that the first joint is fixed at $0\in \R^3$, each joint can rotate with no restrictions with $3$ degress of freedom, and the length $\ell_i$ of each link varies from $a_i$ to $b_i$ with $0<a_i\leq \ell_i\leq b_i$ for $1\leq i\leq k$. We assume $k\ge 2$, or $k=1$ with $a_1<b_1$, and  denote by  $f\colon X\to Y$ the work map associated to this system.

\begin{proposition} (i) The work map $f$ is a submersion.

(ii) If $Y_0\subset Y$ is a sufficiently small open ball then $\tcn(f|_{X_0})=0$ with  $X_0=f^{-1}(Y_0)$.

\end{proposition}

\begin{proof} The configuration space of the system is the compact manifold (with boundary if $a_i< b_i$ for some $i$, that is, if any of the links has variable length),
 $$
 X=\prod_{i=1}^k[a_i,b_i]\times (S^2)^{\times k}.
 $$
The end effector of the robot arm is given by the sum of the vector represented by the links. That is,
$$
Y=\{ \ell_1\overrightarrow{u_1}+\ell_2\overrightarrow{u_2}+\cdots+ \ell_k\overrightarrow{u_k},\,\,a_i\le \ell_i\le b_i,\,\,\overrightarrow{u_i}\in S^2\}\subset \mathbb{R}^3.
$$
It is straightforward to check that the work space $Y$ is either, the closed ball $\overline B(0;r)$ centered at the origin and of radius $r=\sum_{i=1}^k\ell_i$ if $\ell_1\le \sum_{i= 2}^k\ell_i$, or $\overline B(0;r) \setminus B(0;s)$ where $s=\ell_1-\sum_{i=2}^k\ell_i$ otherwise. In any case, the work map is given by
$$
 f(\ell_1,\ldots,\ell_k, \overrightarrow{u_1},\ldots, \overrightarrow{u_k})=\ell_1\overrightarrow{u_1}+\ell_2\overrightarrow{u_2}+\cdots+ \ell_k\overrightarrow{u_k}.
$$
The first assertion follows then by simply taking partial derivatives on $f$. Hence,
 if $Y_0\subseteq Y$ is an open ball with small radius, there is a cross-section $s\colon Y_0\to X_0$ and so  ${\tcn}(f|_{X_0})=0$ by Proposition~\ref{proposition4.4}.

\end{proof}

\begin{remark} (1) Observe that, since $f\times f$ is also a submersion, $Y\times Y$ can be covered by $n+1$ balls (or ``half'' balls) of sufficiently small radius so that $f\times f$  has a local section on each of them. Then, Proposition \ref{anotheruppern} implies that $\tcn(f)\le n$.

(2) It is important to remark that the classical Ereshmann Theorem \cite{e}, by which a proper submersion is a fibration, is not applicable to $f$ in Proposition above. Indeed, $f$ is a proper submersion but between manifolds with boundary. To these manifolds Ereshmann Theorem remains true as long as $f$ preserves the boundary which is not our case. Hence we may not apply Proposition~\ref{proposition4.3} to conclude that, in general, $\tcn(f)=\tc(f)$.

Nevertheless, in some practical situations, and due to constraint conditions on the rotations of the joints and/or on the length of the links, the methodology and statements in~\cite{Meigniez} can be used for checking whether $f|_{X_0}\colon X_0\to Y_0$ is a fibration so that Proposition~\ref{proposition4.3} is applicable. Indeed,
according to~\cite[Theorem A]{Meigniez}, a surjective continuous map $f\colon X\to Y$ is a fibration if and only if it satisfies the following three conditions: it is a homotopic submersion, all vanishing cycles of all dimensions
are trivial, and all emerging cycles of all dimensions are trivial (see op. cit  for explicit definitions of these terms).

\end{remark}

\subsection*{Acknowledgement} We wish to thank Shmuel Weinberger for helpful comments. He proposed the introduction  of the naive version of the topological complexity of a map and its further comparison with the non strict one on the motion planning of a robot arm with $k$ joints.

\end{document}